\documentclass{amsart}
\usepackage{eurosym}
\usepackage{amsthm,amssymb,amsfonts,amsmath,mathrsfs}
\usepackage{dsfont}
\usepackage{graphicx}
\usepackage{enumerate}
\usepackage{hyperref}
\usepackage[nodvipsnames]{xcolor}
\usepackage{soul}
\usepackage{pgfplots}

\newcommand{\notaremm}[1]{\textcolor{blue}{\ifmmode\text{\sout{\ensuremath{#1}}}
\else\sout{#1}\fi }}

\newtheorem{theorem}{Theorem}

\newtheorem{corollary}[theorem]{Corollary}

\newtheorem{definition}[theorem]{Definition}

\newtheorem{lemma}[theorem]{Lemma}

\newtheorem{proposition}[theorem]{Proposition}
\newtheorem{remark}[theorem]{Remark}

\renewcommand{\S}{\mathbb{S}}

\DeclareMathOperator{\Var}{Var}

\input{tcilatex}

\begin{document}
\title{Quadratic response and speed of convergence of invariant measures in
the zero-noise limit.}
\author{Stefano Galatolo}
\address{Dipartimento di Matematica, Universit\`a di Pisa, Largo Bruno
Pontecorvo 5, 56127 Pisa, Italy}
\email{stefano.galatolo@unipi.it}
\urladdr{http://pagine.dm.unipi.it/~a080288/}
\author{Hugo Marsan}
\address{---}
\email{hugo.marsan@ens-paris-saclay.fr}
\date{\today }

\begin{abstract}
We study the stochastic stability in the zero-noise limit from a
quantitative point of view.

We consider smooth expanding maps of the circle, perturbed by additive
noise. We show that in this case the zero-noise limit has a quadratic speed
of convergence, as conjectured by Lin, in 2005, after numerical experiments
(see \cite{Lin}). This is obtained by providing an explicit formula for the
first and second term in the Taylor's expansion of the response of the
stationary measure to the small noise perturbation. These terms depend on
important features of the dynamics and of the noise which is perturbing it,
as its average and variance.

We also consider the zero-noise limit from a quantitative point of view for
piecewise expanding maps showing estimates for the speed of convergence in
this case.
\end{abstract}

\maketitle

\section{Introduction}

Deterministic dynamical systems are often used as models of physical and
natural phenomena despite the ubiquitous presence in nature of small random
perturbations or fluctuations. It is natural to study the robustness of the
deterministic model to such random perturbations and which of the aspects of
the deterministic dynamics are stable under small random perturbations. In
this paper we consider the many important aspects of the statistical
behavior of the system which are encoded in its invariant measures. We study
hence quantitatively how these measures change when the system is perturbed
by the adding of a small quantity of noise, in the so called zero-noise
limit. More precisely, we study this limit and its speed of convergence from
a quantitative point of view, also considering first and second order terms
in this convergence. We will see that these terms depend on important
features of the dynamics and of the noise which is perturbing it, as its
average and variance.

Let $S_{0}:=(X,T)$ be a discrete time deterministic dynamical system where $%
X\ $is a metric space and$\ \ T:X\rightarrow X$ \ is a Borel measurable map.
It is well known that $(X,T)$ can have several invariant measures, let us
consider one of these measures and denote it by\ $\mu _{0}$.

Suppose now we perturb the system at each iteration by the adding of a small
quantity of noise whose amplitude is expressed by a certain parameter $%
\delta $ $\in \lbrack 0,\overline{\delta })$ obtaining a family of random
systems $\{S_{\delta }\}_{\delta \in \lbrack 0,\overline{\delta })}$ (these
systems will formally be defined as suitable random dynamical systems, a
precise definition will be given in Section \ref{QR2}). Suppose $\{\mu
_{\delta }\}_{\delta \in \lbrack 0,\overline{\delta })}$ are stationary
measures for $S_{\delta }.$ It is natural to investigate under which
assumptions one may have%
\begin{equation*}
\lim_{\delta \rightarrow 0}\mu _{\delta }= \mu _{0}.
\end{equation*}%
In this case the system and the measure $\mu _{0}$ are said to be \emph{%
statistically stable} under small noise perturbations (or in the zero-noise
limit). An invariant measure of a deterministic model which is stable under
the small random perturbations which are present in nature is a measure that
can be observed in the real phenomenon behind the model. For this reason
this zero-noise limit was proposed by A. N. Kolmogorov as a tool to select
the \emph{physically meaningful }measures among the a priori many invariant
measures of a deterministic system (see e.g. \cite{ER}, \cite{Y}).

The statistical stability for the zero-noise limit (also called stochastic
stability) was proved for several classes of systems, starting from
uniformly hyperbolic ones to many interesting cases of non-uniform
hyperbolic behavior (\cite{Ki},\cite{Y2}, \cite{BK1}, \cite{BK2}, \cite{BY}, 
\cite{met}, \cite{Sh}, \cite{AV}, \cite{LSV}, \cite{AT}, \cite{AA}, \cite{A}%
, \cite{AK}, \cite{S}, \cite{BV}).

The mere existence of the zero-noise limit gives a qualitative information
on the behavior of the system under perturbation. In practice it can be
useful to have quantitative information on this convergence, both on the
speed of the convergence and on the "direction" of change of the invariant
measure after the perturbation. In \cite{Lin} several numerical experiments
have been done to estimate the speed of convergence in the limit,
conjecturing a quadratic speed in the case of smooth expanding maps and
linear speed for the piecewise expanding and hyperbolic case. Other
exponents have been conjectured in cases of weakly chaotic, non-uniformly
hyperbolic systems.

In this paper we will consider these kinds of questions, investigating both
quantitative estimates for the speed of the convergence and the direction of
change of the invariant measure of the system under perturbation. This is
strongly related to the linear response theory, although in this case we
will not be only interested in the linear term in the response of the system
to the perturbation, but also in the higher order terms, and in particular
to the quadratic one.

The \emph{Linear Response} means to quantify the response of the system when
submitted to a certain infinitesimal perturbation as a derivative. For
example, if one is interested in the linear response of the stationary
measure of the system, we will consider the derivative of the invariant
measure of interest with respect to the perturbation.

More precisely, \ let $\{S_{\delta }\}_{\delta \in \lbrack 0,\overline{%
\delta })}$ as above be the family of systems arising by some small
perturbation of the initial system $S_{0}$ with stationary measures $\{\mu
_{\delta }\}_{\delta \in \lbrack 0,\overline{\delta })}$. The linear
response of the invariant measure $\mu _{0}$ of $S_{0}$ under the given
perturbation is defined by the limit 
\begin{equation}
\dot{\mu}:=\lim_{\delta \rightarrow 0}\frac{\mu _{\delta }-\mu _{0}}{\delta }
\label{LRidea}
\end{equation}%
where the meaning of this convergence can vary from system to system. \ In
some systems and for a given perturbation, one may get $L^{1}$-convergence
for this limit; in other systems or for other perturbations one may get
convergence in weaker or stronger topologies. The linear response to the
perturbation hence represents the first order term of the response of a
system to the perturbation and in this case, a linear response formula can
be written: 
\begin{equation}
\mu _{\delta }=\mu _{0}+\dot{\mu}\delta +o(\delta )  \label{lin}
\end{equation}%
which holds in some weaker or stronger sense. \ We remark that given an
observable function $c:X\rightarrow \mathbb{R}$, if the convergence in %
\eqref{LRidea} is strong enough with respect to the regularity of $c$, we get

\begin{equation}
\lim_{t\rightarrow 0}\frac{\int \ c\ d\mu _{t}-\int \ c\ d\mu _{0}}{t}=\int
\ c\ d\dot{\mu}  \label{LRidea2}
\end{equation}%
showing how the linear response controls the behavior of observable
averages. For instance the convergence in \eqref{LRidea2} hold when $c\in
L^{\infty }$ and the convergence of the linear response is in $L^{1}$.

Once the first order (the linear part) of the response of a system to a
perturbation is understood, it is natural to study further orders. The
second order of the response may then be related to the second derivative
and to other natural questions, as convexity aspects of the response of the
system under perturbation, or the stability of the first order response.
Hence, if the Linear Response $\dot{\mu}$ represents the first order term of
the response (see \eqref{lin}), the Quadratic Response $\ddot{\mu}$ will
represent the second order term of this response, analogous to the second
derivative in usual Taylor's expansion:%
\begin{equation}
\mu _{\delta }=\mu _{0}+\dot{\mu}\delta +\frac{1}{2}\ddot{\mu}\delta
+o(\delta ^{2}).  \label{Quad}
\end{equation}

We refer to \cite{BB} for \ a recent survey on linear response for
deterministic systems and perturbations and to the introduction of \cite{GS}
for a very recent survey in the case of response for random systems and
higher order terms in the response of a system to deterministic or random
perturbations. Focusing on zero-noise limits, we point out the paper \cite%
{GL}, where among other results, linear and high order response are proven
for deterministic perturbations and zero-noise limits of uniformly
hyperbolic systems (see also \cite{Li2} for some earlier examples of linear
response in the zero-noise limit for expanding maps and \cite{BGNN} for
rigorous numerical methods for its approximation including an example of
zero-noise limit).

In the paper \cite{GS}, a relatively simple and quite general approach to the
first and second order terms in the response of a system to  perturbations is proposed and applied to deterministic perturbations of deterministic systems and perturbations of random systems.
In Subsection \ref{secgenres} we recall the main general results of \cite{GS}%
. We then apply it to the zero-noise limit, providing precise quantitative
information on the convergence of the zero-noise limit and proving some of
the conjectures suggested by the numerical experiments and the heuristic
exposed in \cite{Lin}, in particular considering zero-noise limits of
expanding and piecewise expanding maps.

In the literature, the general approach to this problem is often based on
considering the family of transfer operators $\{L_{\delta }\}_{\delta \in
\lbrack 0,\overline{\delta })}$ associated to the dynamical system and its
perturbations, remarking that invariant and stationary measures are fixed
points of this family of operators. Quantitative perturbative statements
about these operators and its spectral picture will hence give information
on the perturbation of invariant measures. In this paper we use these tools
to study the zero-noise limit from a quantitative point of view in two main
cases: \emph{smooth expanding maps} and \emph{piecewise expanding maps} of
the circle. In the following two subsections we enter in more details about
our main results in these two cases.

\noindent \textbf{Smooth expanding maps, response and zero-noise limit. }We
consider a smooth expanding map $T\in C^{8}(\mathbb{S}^{1}\rightarrow 
\mathbb{S}^{1})$, with its associated transfer operator $L_{T}:SM(\mathbb{S}%
^{1})\rightarrow SM(\mathbb{S}^{1}),$ \ where $SM(\mathbb{S}^{1})$ is the
space of finite Borel measures with sign on $\mathbb{S}^{1},$ defined by%
\begin{equation*}
(L_{T}(\mu ))(A)=\mu (T^{-1}(A))
\end{equation*}%
for each signed measure $\mu \in SM(\mathbb{S}^{1}).$ $L_{T}$ is also called
the transfer operator associated to $T$ or pushforward map associated to $T$%
. We consider an i.i.d. random perturbation distributed according to a
kernel $\rho \in BV([-1,1])$. $\forall \delta \in \lbrack 0,\overline{\delta 
})$ we denote by $\rho _{\delta }$ the rescaling of $\rho $ with amplitude $%
\delta $ by 
\begin{equation*}
\rho _{\delta }(x)=\frac{1}{\delta }\rho \left( \frac{x}{\delta }\right) .
\end{equation*}%
The transfer operator associated to the randomly perturbed map is then
defined as 
\begin{equation*}
L_{\delta }=\rho _{\delta }\ast L_{T}
\end{equation*}%
where $\ast $ stands for the ordinary convolution operator on $\mathbb{S}%
^{1}.$ Note that we can extend this definition to $\delta =0$ with $\rho
_{0}=\delta _{0}$ the Dirac mass. It can be proved that (see Section \ref%
{secQR}) each operator $L_{\delta }$ has a unique fixed point $h_{\delta }$
in the Sobolev space $W^{7,1}(\mathbb{S}^{1})$ and hence $h_{\delta }$ is the
stationary measure of the perturbed system.

The idea is to prove that this family of operators admits a linear, and even
quadratic response when $\delta$ tends to $0$. In particular, we prove and
extend a result conjectured in \cite{Lin}, in which the author predicted a
convergence of order $\delta ^{2}$. We will precise the coefficients of the
order two Taylor's expansion of this zero-noise limit proving the following
theorem:

\begin{theorem}[Quadratic response in the zero-noise limit for a smooth
expanding map]
\label{thm:quadrespsmooth} The map $\delta \mapsto h_{\delta }$ has an order
two Taylor's expansion at $\delta =0$, with%
\begin{equation}
\left\Vert \frac{h_{\delta }-h_{0}}{\delta ^{2}}-\frac{\sigma^2(\rho )}{2}%
(Id-L_{T})^{-1}h_{0}^{\prime \prime }\right\Vert _{W^{{1,1}}}\underset{%
\delta \rightarrow 0}{\longrightarrow }0.
\end{equation}%
with $\sigma^2(\rho )=\int_{-1}^{1}x^{2}\rho (x)dx$.
\end{theorem}

Next Section \ref{secQR} is essentially devoted to the proof of this result.
We prove the theorem by the application of some general linear and quadratic
response statements we recall in subsection \ref{secgenres}. In subsection %
\ref{verif} we verify the several assumptions needed to apply those
theorems, completing the proof at the end of Section \ref{secQR}.

\noindent \textbf{Piecewise expanding maps, quantitative stability and
zero-noise limit. }We have seen that for smooth expanding maps there is a
quadratic speed of convergence in the zero-noise limit. This depend both on
the smoothness and on the strong chaoticity of the system. When having less
regularity, the speed of convergence changes. In the second part of the
paper we consider indeed piecewise expanding maps, allowing discontinuities.
In this case we have systems still having strong chaoticity, and exponential
decay of correlations, but the speed of convergence in the zero limit is of
order 1. We prove in fact the following

\begin{proposition}
\label{prop:linresppiecewise} Let $T:\mathbb{S}^{1}\rightarrow \mathbb{S}%
^{1} $ be a piecewise expanding map having no periodic turning points (see
Section \ref{sec:pwexp} for the precise definitions). Let us suppose we
perturb the associated dynamical system with noise of amplitude $\delta $ as
above, let $L_{\delta }$ be the associated transfer operators and let $%
h_{\delta }$ be a family of invariant measures for $L_{\delta }$ then $%
h_{\delta }\in Lip[0,1]$ and there is $C\geq 0$ such that for each $\delta
\in \lbrack 0,\overline{\delta })$%
\begin{equation*}
||h_{0}-h_{\delta }||_{L^{1}}\leq C\delta \log \delta .
\end{equation*}%
Furthermore, 
there are examples of piecewise expanding maps (with periodic turning
points) for which there is a constant $C^{\prime }$ such that for each $%
\delta \in \lbrack 0,\overline{\delta })$, 
\begin{equation*}
||h_{0}-h_{\delta }||_{L^{1}}\geq C^{\prime }\delta .
\end{equation*}
\end{proposition}

\section{Quadratic response and the zero-noise limit of expanding maps\label%
{secQR}}

In this section we consider the zero-noise limit of expanding maps on the
circle. We get precise estimates on the speed of convergence of this limit,
proving Theorem \ref{thm:quadrespsmooth}.

In this section we will consider maps $T:\mathbb{S}^{1}\rightarrow \mathbb{S}%
^{1}$ satisfying the following assumptions

\begin{enumerate}
\item $T\in C^{8}(\mathbb{S}^{1}\mathbb{)},$

\item $|T^{\prime }(x)|\geq \alpha ^{-1}>1$ $\forall x\in \mathbb{S}^{1}$.
\end{enumerate}

To $T$ is associated a linear map $L_{T}:SM(\mathbb{S}^{1})\rightarrow SM(%
\mathbb{S}^{1}),$ \ where $SM(\mathbb{S}^{1})$ is the space of Borel
measures with sign on $\mathbb{S}^{1},$ defined by%
\begin{equation*}
(L_{T}(\mu ))(A)=\mu (T^{-1}(A))
\end{equation*}%
for each signed measure $\mu \in SM(\mathbb{S}^{1}).$ \ $L_{T}$ is also
called the transfer operator associated to $T$ or pushforward map associated
to $T$.

We consider a perturbation of this transfer operator by adding to the
deterministic dynamics generated by $T$ a random independent and identically
distributed perturbation: the noise. In other words we consider a random
dynamical system, corresponding to the stochastic process $(X_{n})_{n\in 
\mathbb{N}}$ defined by 
\begin{equation}
X_{n+1}=T(X_{n})+\Omega _{n}\mod 1  \label{eq:syswaddnoise}
\end{equation}%
where $(\Omega _{n})_{n\in \mathbb{N}}$ are i.i.d random variables
distributed according to a probability density $\rho _{\delta }$ (the noise
kernel) where $\delta $ represent the "size" of the perturbation. We suppose
that $\rho _{\delta }$ is obtained by rescaling a certain distribution $\rho 
$ $\in BV([-1,1]),$ \ as follows 
\begin{equation}
\rho _{\delta }(x)=\frac{1}{\delta }\rho \left( \frac{x}{\delta }\right)
\label{eq:kerdef}
\end{equation}%
for each $\ \delta \in (0,1].$ \ The (annealed) transfer operator associated
to the perturbed random system is then defined by%
\begin{equation}
L_{\delta }=\rho _{\delta }\ast L_{T}  \label{opdef}
\end{equation}%
where $\ast $ \ is the convolution operator on $\mathbb{S}^{1}$ (see e.g.%
\cite{Viana}, Section 5 or \cite{GMN}, Section 8 for more details on the
definition of the annealed transfer operator).

Remark that for each $f\in L^{1}(\mathbb{S}^{1})$, one has for almost every $%
x\in \mathbb{S}^{1}$: 
\begin{equation}
((\rho _{\delta }-\delta _{0})\ast f)(x)=\frac{1}{\delta }\int_{-\delta
}^{\delta }\rho \left( \frac{y}{\delta }\right)
f(x-y)dy-f(x)=\int_{-1}^{1}\rho (z)f(x-\delta z)dz-f(x).  \label{eq:kernel}
\end{equation}%
Also, one can remark that $\rho $ being a zero-average probability kernel,
it verifies 
\begin{equation}
\int_{-1}^{1}\rho (z)dz=1\qquad \int_{-1}^{1}\rho (z)zdz=0\qquad
\int_{-1}^{1}\rho (z)z^{2}dz:=\sigma ^{2}(\rho ).
\end{equation}

To keep the notation compact we will denote $L_{0}:=L_{T}$. We remark that
the invariant measures of the map $T$ are fixed points of $L_{0}$ and the
stationary measures of the random system constructed by the adding of the
noise are fixed points of $L_{\delta }$. We are interested in the properties
of these measures and how they vary as $\delta $ goes to $0$. Their
characterization as fixed points of $L_{\delta }$ will be sufficient for our
purposes. We recall that in the case we are considering (expanding maps)
there can be a large set of invariant measures for the deterministic map $T$
but only one which is absolutely continuous with respect to the Lebesgue
measure. On the other hand the stationary measures for $L_{\delta }$ will
always be absolutely continuous. This is well known (see e.g. \cite{notes})
but it can also be easily derived from the regularization estimates we prove
in the following.

\subsection{General linear response and quadratic response results\label%
{secgenres}}

In this subsection we state some general results from \cite{GS} about linear
and quadratic response of fixed points of Markov operators under suitable
perturbations. These results will be applied to our zero-noise limit and to
the operators $L_{\delta }$ to get precise estimates on the speed of
convergence of $h_{\delta }$ to $h_{0}.$

In the following we consider four normed vectors spaces of signed Borel
measures on $\mathbb{S}^{1}.$ The spaces $(B_{ss},\Vert ~\Vert
_{ss})\subseteq (B_{s},\Vert ~\Vert _{s})\subseteq (B_{w},\Vert ~\Vert
_{w})\subseteq (B_{w},\Vert ~\Vert _{ww})\subseteq BS(\mathbb{S}^{1})$ with
norms satisfying%
\begin{equation*}
\Vert ~\Vert _{ww}\leq \Vert ~\Vert _{w}\leq \Vert ~\Vert _{s}\leq \Vert
~\Vert _{ss}.
\end{equation*}

We will assume that the linear form $\mu \rightarrow \mu (\mathbb{S}^{1})$
is continuous on $B_{i}$, for $i\in \{ss,s,w,ww\}$. Since we will consider
Markov operators\footnote{%
A Markov operator is a linear operator preserving positive measures and such
that for each positive measure $\mu $, it holds $[L(\mu )](X)=\mu (X)$.}
acting on these spaces, the following (closed) spaces $V_{ss}\subseteq
V_{s}\subseteq V_{w}\subseteq V_{ww}$ of zero average measures will play an
important role. We define $V_{i}$ as:%
\begin{equation*}
V_{i}:=\{\mu \in B_{i}\mid \mu (\mathbb{S}^{1})=0\}
\end{equation*}%
where $i\in \{ss,s,w,ww\}$. Suppose hence we have a one parameter family of
such Markov operators $L_{\delta }.$ The following is a general statement
establishing  linear response for suitable perturbations of such operators.

\begin{theorem}[Linear Response]
\label{thm:linresp} Suppose that the family of bounded Markov operators $%
L_{\delta }:B_{i}\rightarrow B_{i},$ where $i\in \{ss,s,w\}$ satisfy the
following:

\begin{itemize}
\item[(LR1)] (regularity bounds) for each $\delta \in \left[ 0,\overline{%
\delta }\right) $ there is $h_{\delta }\in B_{ss}$, a probability measure
such that $L_{\delta }h_{\delta }=h_{\delta }$. Furthermore, there is $M\geq
0$ such that for each $\delta \in \left[ 0,\overline{\delta }\right) $ 
\begin{equation*}
\Vert h_{\delta }\Vert _{ss}\leq M.
\end{equation*}

\item[(LR2)] (convergence to equilibrium for the unperturbed operator) There
is a sequence $a_n\to 0$ such that for each $g\in V_{ss}$%
\begin{equation*}
\Vert L_{0}^{n}g\Vert _{s}\leq a_n||g||_{ss};
\end{equation*}

\item[(LR3)] (resolvent of the unperturbed operator) $(Id-L_{0})^{-1}:=%
\sum_{i=0}^{\infty }L_{0}^{i}$ is a bounded operator $V_{w}\rightarrow V_{w} 
$.

\item[(LR4)] (small perturbation and derivative operator) There is $K\geq 0$
such that $\left\vert |L_{0}-L_{\delta }|\right\vert _{B_{s}\rightarrow
B_{w}}\leq K\delta ,$ and $\left\vert |L_{0}-L_{\delta }|\right\vert
_{B_{ss}\rightarrow B_{s}}\leq K\delta $. There is ${\dot{L}h_{0}\in V_{w}}$
such that%
\begin{equation}
\underset{\delta \rightarrow 0}{\lim }\left\Vert \frac{(L_{\delta}-L_{0})}{%
\delta }h_{0}-\dot{L}h_{0}\right\Vert _{w}=0.  \label{derivativeoperator}
\end{equation}
\end{itemize}

Then we have the following Linear Response formula%
\begin{equation}
\lim_{\delta \rightarrow 0}\left\Vert \frac{h_{\delta }-h_{0}}{\delta }%
-(Id-L_{0})^{-1}\dot{L}h_{0}\right\Vert _{w}=0.  \label{linresp}
\end{equation}
\end{theorem}

The following is an abstract response result for the second derivative.

\begin{theorem}[Quadratic term in the response]
\label{thm:quadresp} Let $(L_{\delta })_{\delta \in \lbrack 0,\overline{%
\delta }]}:B_{i}\rightarrow B_{i}$, $i\in \{ss,...,ww\}$ be a family of
Markov operators as in the previous theorem. \ Assume furthermore that:

\begin{enumerate}
\item[(QR1)] The derivative operator $\dot{L}$ admits a bounded extension $%
\dot{L}:B_{w}\rightarrow V_{ww}$, such that 
\begin{equation}
\left\Vert \frac{1}{\delta }(L_{\delta}-L_{0})-\dot{L}\right\Vert
_{w\rightarrow ww}\underset{\delta \rightarrow 0}{\longrightarrow }0.
\label{eq:unifcvderivop}
\end{equation}

\item[(QR2)] There exists a "second derivative operator" at $h_{0}$, i.e. $\ 
\ddot{L}h_{0}\in V_{ww}$ such that 
\begin{equation}
\left\Vert \dfrac{(L_{\delta }-L_{0})h_{0}-\delta \dot{L}h_{0}}{\delta ^{2}}-%
\ddot{L}h_{0}\right\Vert _{ww}\underset{{\delta }\rightarrow 0}{%
\longrightarrow }0.  \label{def:2ndderivop}
\end{equation}

\item[(QR3)] The resolvent operator $(Id-L_{0})^{-1}$ admits a bounded
extension as an operator $V_{ww}\rightarrow V_{ww}$.
\end{enumerate}

Then one has the following: the map $\delta \in \lbrack 0,\overline{\delta }%
]\mapsto h_{\delta }\in B_{ss}$ has an order two Taylor's expansion at $%
\delta =0$, with 
\begin{equation}
\left\Vert \frac{h_{\delta }-h_{0}-\delta (Id-L_{0})^{-1}\dot{L}h_{0}}{%
\delta ^{2}}-(Id-L_{0})^{-1}\left[ \ddot{L}h_{0}+\dot{L}(Id-L_{0})^{-1}\dot{L%
}h_{0}\right] \right\Vert _{ww}\underset{{\delta }\rightarrow 0}{%
\longrightarrow }0.  \label{eq:quadresp}
\end{equation}
\end{theorem}

Given the family of transfer operators $L_{\delta }$ defined at \ref{opdef},
we will apply these response results using the sequence of stronger and
weaker spaces%
\begin{equation*}
W^{7,1}(\mathbb{S}^{1})\subset W^{5,1}(\mathbb{S}^{1})\subset W^{3,1}(%
\mathbb{S}^{1})\subset W^{1,1}(\mathbb{S}^{1})
\end{equation*}

where $W^{k,1}$ stands for the Sobolev space of functions having the $k^{th}$
derivative in $L^{1}$ (see \cite{ADA} for an introduction to these spaces).

In the following subsection we verify the assumptions needed to apply
theorems \ref{thm:linresp} \ and \ref{thm:quadresp}. Theorem \ref%
{thm:quadrespsmooth}\ will be proved at the end of the section.

\subsection{Verifying the assumptions in the general response theorems\label%
{verif}.}

In this subsection we verify the assumptions needed to apply theorems \ref%
{thm:linresp} \ and \ref{thm:quadresp}. \ First we verify the spectral gap
and existence of the resolvent assumptions for the unperturbed system. This
is somewhat well known for circle expanding maps. However for completeness
we recall the main steps of this construction in subsection \ref{LR2LR3QR3}.
In subsection \ref{LR1} we prove a uniform Lasota Yorke inequality, to
verify the assumption LR1. In subsection \ref{LR4QR1} we compute the first
derivative operator associated to the small-noise perturbation, verifying
assumptions LR4 and QR1. In subsection \ref{QR2} we compute the second
derivative operator, verifying assumption QR2.

\subsubsection{Spectral gap and resolvent for the unperturbed operator
(verifying LR2, LR3 and QR3)\label{LR2LR3QR3}.}

In this section we consider the tranfer operator $L_{0}$ of the unperturbed
system acting on our Sobolev spaces and verify the convergence to
equilibrium and the existence of the resolvent operator $(Id-L_{0})^{-1}$on
the weak and weakest spaces $W^{3,1}(\mathbb{S}^{1}),$ $W^{1,1}(\mathbb{S}%
^{1})$ as required in assumptions LR2, LR3 and QR3. Since we are considering
the transfer operator associated to an expanding map of the circle, these
results are nowadays not surprising \textbf{(}see e.g. \cite{Li2}\textbf{). }%
The results follow from a standard construction, in which one can get
information on the spectrum of $L_{0}$ acting on these spaces from the
existence of a Lasota Yorke inequality on suitable functional spaces which
are compactly embedded each other. As there are some variants of this
contruction, for completeness we briefly recall some precise statements
which we can apply to our case.

The following theorem (see Theorem 9 in \cite{GS} for a proof of the
statement in this form) is a version of a classical tool to obtain spectral
gap in systems satisfying a Lasota Yorke inequality. It allows one to
estimate the contraction rate of zero average measures, and imply spectral
gap when applied to Markov operators. Let us consider a Markov operator $%
L_{0}$ acting on two normed vector spaces of complex or signed measures $%
(B_{s},\Vert ~\Vert _{s}),~(B_{w},\Vert ~\Vert _{w}),$ $B_{s}\subseteq B_{w}$
with $\Vert ~\Vert _{s}\geq \Vert ~\Vert _{w}$. We furthermore assume that $%
\mu \mapsto \mu (X)$ is continuous in the $\Vert ~\Vert _{s}$ and $\Vert
~\Vert _{w}$ topologies, and let $V_{i}:=\{\mu \in B_{i},\mu (X)=0\}$, $i\in
\{w,s\}$.

\begin{theorem}
\label{gap}Suppose:

\begin{enumerate}
\item (Lasota Yorke inequality). For each $g\in B_{s}$%
\begin{equation*}
\|L_0^{n}g\|_{s}\leq A\lambda _{1}^{n}\|g\|_{s}+B\|g\|_{w};
\end{equation*}

\item (Mixing) for each $g\in V_{s}$, it holds 
\begin{equation*}
\lim_{n\rightarrow \infty }\|L_0^{n}g\|_{w}=0;
\end{equation*}

\item (Compact inclusion) The image of the closed unit ball in $B_{s}$ under 
$L_0$ is relatively compact in $B_{w}$.

\end{enumerate}

Under these assumptions, we have

\begin{description}
\item[a] $L_{0}$ admits a unique fixed point in $h\in B_{s}$, satisfying $%
h(X)=1$.

\item[b] There are $C>0,\rho <1$ such that for all $f\in V_{s}$ and $m$
large enough, 
\begin{equation}
\Vert L_{0}^{m}f\Vert _{s}\leq C\rho ^{m}\Vert f\Vert _{s}.  \label{gap2}
\end{equation}

\item[c] The resolvent $(Id-L_{0})^{-1}:$ $V_{s}\rightarrow V_{s}$ is
defined and continuous.
\end{description}
\end{theorem}

To apply this result to our case we recall the following (a proof can be
found in \cite{GS}).

\begin{lemma}
\label{Lemsu}A $C^{k+1}$ expanding map on $\mathbb{S}^{1}$ satisfies a
Lasota-Yorke inequality on $W^{k,1}(\S ^{1})$: there is $\alpha <1$, $%
A_{k},~B_{k}\geq 0$ such that 
\begin{equation*}
\left\{ \begin{aligned} &\|L^nf\|_{W^{k-1,1}}\leq A_k\|f\|_{W^{k-1,1}}\\
&\|L^{n}f\|_{W^{k,1}}\leq \alpha^{kn}\|f\|_{W^{k,1}}+B_k\|f\|_{W^{k-1,1}}
\end{aligned}\right. .
\end{equation*}
\end{lemma}

By the compact immersion of $W^{k,1}$ in $W^{s,1}$ when $k\geq s$ (see \cite%
{ADA}) and the well-known fact that an expanding map satisfies the mixing
assumption, we can apply Theorem \ref{gap} to our transfer operator of a $%
C^{8}$ expanding map and deduce the following result.

\begin{proposition}
\label{propora}For each $k\in \{1,2,\cdots ,7\}$ there are $C>0,\rho <1$
such that for each $g\in V_{k}$ it holds 
\begin{equation*}
\Vert L^{n}g\Vert _{W^{k,1}}\leq C\rho ^{n}\Vert g\Vert _{W^{k,1}}.
\end{equation*}%
In particular, the resolvent $(Id-L)^{-1}=\sum_{i=0}^{\infty }L^{i}$ is a
well-defined and bounded operator on $V_{k}$.
\end{proposition}

This is enough to verify assumptions LR2, LR3 and QR3 in our case.

\subsubsection{Uniform Lasota Yorke inequalities in the zero-noise limit
(verifying LR1)\label{LR1}}

In order to prove the assumption LR1, we use Theorem \ref{gap} and show a
uniform Lasota-Yorke inequality.

\begin{lemma}
\label{lem:contraction} $\forall k \geq 0, f\in W^{k,1}(\mathbb{S}^1),$ 
\begin{equation*}
\|L_\delta f\|_{W^{k,1}} \leq \|L_T f\|_{W^{k,1}}.
\end{equation*}
\end{lemma}

\begin{proof}
We first prove the general statement: $\forall f\in L^1(\mathbb{S}^1)$, $%
\|\rho_\delta * f\|_{L^1} \leq \|f\|_{L^1}$. $\forall x \in \mathbb{S}^1 $,
we have 
\begin{equation*}
\rho_\delta * f(x) = \int _{-\delta}^\delta \rho_\delta(y) f(x-y) dy.
\end{equation*}
Hence 
\begin{equation*}
\|\rho_\delta* f\|_{L^1} \leq \int_{\mathbb{S}^1}\int _{-\delta}^\delta
\rho_\delta (y) |f(x-y)| dydx = \int _{-\delta}^\delta \rho_\delta(y)dy
\times \|f\|_{L^1} = \|f\|_{L^1}.
\end{equation*}
Using the fact that $\forall i\geq 0$, 
\begin{equation*}
(L_\delta f)^{(i)} = (\rho_\delta * L_Tf)^{(i)} = \rho_\delta * \left(
L_Tf\right)^{(i)},
\end{equation*}
we then have 
\begin{equation*}
\|(L_\delta f)^{(i)}\|_{L^1} \leq \|\left( L_Tf \right)^{(i)}\|_{L^1}.
\end{equation*}
Hence 
\begin{equation*}
\|L_\delta f\|_{W^{k,1}} = \sum_{i=0}^k \|(L_\delta f)^{(i)}\|_{L^1} \leq
\|L_T f\|_{W^{k,1}}.
\end{equation*}
\end{proof}

\begin{lemma}
\label{lem:LYNoise} For $k\geq 1$, the family $(L_\delta)$ verifies a
uniform Lasota-Yorke inequality on $W^{k,1}(\mathbb{S}^1)$, which is: there
is $\alpha <1, C_k, D_k \geq 0$ such that: 
\begin{equation*}
\left\{ \begin{aligned} &\|L_\delta^nf\|_{W^{k-1,1}}\leq
C_k\|f\|_{W^{k-1,1}}\\ &\|L_\delta^{n}f\|_{W^{k,1}}\leq
\alpha^{kn}\|f\|_{W^{k,1}}+D_k\|f\|_{W^{k-1,1}} \end{aligned}\right. .
\end{equation*}
\end{lemma}

\begin{proof}
We will prove this lemma by induction on $k\geq 1$. $L_T$ is a contraction
on $L^1$: using Lemma \ref{lem:contraction}, it is also the case for $%
L_\delta$, proving the power-boundedness on $L^1$ (with $C_1 = 1$).

Then, using Lemma \ref{Lemsu}, we know that there is a $B_1 \geq 0$ such
that 
\begin{equation*}
\|L_\delta f\|_{W^{1,1}} \leq \|L_T f\|_{W^{1,1}} \leq
\alpha\|f\|_{W^{1,1}}+ B_1 \|f\|_1.
\end{equation*}
Applying this inequality to $L_\delta^2 f = L_\delta(L_\delta f)$ gives us 
\begin{align*}
\|L_\delta^2 f\|_{W^{1,1}} &\leq \alpha\|L_\delta f\|_{W^{1,1}} + B_1
\|L_\delta f\|_{1} \\
&\leq \alpha^2 \|f\|_{W^{1,1}} + \alpha B_1 \|f\|_{1} + B_1 \|L_\delta
f\|_{1}.
\end{align*}
We can then iterate: 
\begin{align*}
\|L_\delta^n f\|_{W^{1,1}} &\leq \alpha^n \|f\|_{W^{1,1}} + B_1
\sum_{i=0}^{n-1} \alpha^i \|L_\delta^{n-1-i} f\|_{1} \\
&\leq \alpha^n \|f\|_{W^{1,1}} + \frac{B_1 C_1}{1-\alpha} \|f\|_{1}
\end{align*}
giving us the property for $k=1$, with $C_1 = 1$ and $D_1 = \frac{B_1 C_1}{%
1-\alpha}$.

The induction is then analogous to the base case. Using the induction
hypothesis on $k-1$, more precisely the Lasota-Yorke inequality, we have the
power-boundedness of $L_\delta$ on $W^{k-1,1}$: 
\begin{equation*}
\|L^n_\delta f\|_{W^{k-1,1}} \leq (1 + D_{k-1})\|f\|_{W^{k-1,1}}.
\end{equation*}
Hence $C_k = 1+D_{k-1}$. We can then use again Lemma \ref{Lemsu} to have the
first inequality 
\begin{equation*}
\|L_\delta f\|_{W^{k,1}} \leq \alpha\|f\|_{W^{k,1}}+ B_k \|f\|_{W^{k-1}},
\end{equation*}
which we can iterate to 
\begin{equation*}
\|L_\delta^n f\|_{W^{k,1}} \leq \alpha^n \|f\|_{W^{k,1}} + B_k
\sum_{i=0}^{n-1} \alpha^i \|L_\delta^{n-1-i} f\|_{W^{k-1,1}}.
\end{equation*}
We can finally use the power-boundedness result we just proved to conclude: 
\begin{equation*}
\|L_\delta^n f\|_{W^{k,1}} \leq \alpha^n \|f\|_{W^{k,1}} + \frac{B_kC_k}{%
1-\alpha} \|f\|_{W^{k-1,1}}.
\end{equation*}
Hence the result for $k$, with $C_k = (1+D_{k-1})$ and $D_k =\frac{B_k C_k}{%
1-\alpha}$.
\end{proof}

We can then extend this inequality to our spaces, $W^{2k+1,1}(\mathbb{S}^1)$.

\begin{corollary}
For $k \geq 2$, the family $(L_\delta)$ verifies a uniform Lasota-Yorke
inequality on $W^{k,1}(\mathbb{S}^1) \subset W^{k-2,1}(\mathbb{S}^1)$: there
is $\alpha <1, E_k, F_k \geq 0$ such that 
\begin{equation*}
\|L^n_\delta f\|_{W^{k,1}} \leq \alpha^{\frac{k-1}{2}n}E_k\|f\|_{W^{k,1}} +
F_k\|f\|_{W^{k-2,1}}.
\end{equation*}
\end{corollary}

\begin{proof}
By using Lemma \ref{lem:LYNoise}, we have that $\forall n,p \geq 0$, 
\begin{align*}
\|L^{n+p}_\delta f\|_{W^{k,1}} &\leq \alpha^{kn}\|L^p_\delta f\|_{W^{k,1}} +
D_k \|L^p_\delta f\|_{W^{k-1,1}} \\
&\leq \alpha^{kn}C_{k+1}\|f\|_{W^{k,1}} + D_k\left(
\alpha^{(k-1)p}\|f\|_{W^{k-1,1}} + D_{k-1}\|f\|_{W^{k-2,1}}\right) \\
&\leq \left(\alpha^{kn}C_{k+1} + \alpha^{(k-1)p}D_k\right)\|f\|_{W^{k,1}} +
D_k D_{k-1}\|f\|_{W^{k-2,1}}.
\end{align*}
So in the case $p=n$ (an even exponent), we have 
\begin{align*}
\|L^{2n}_\delta f\|_{W^{k,1}} &\leq \alpha^{(k-1)n}\left(\alpha^{n}C_{k+1} +
D_k\right)\|f\|_{W^{k,1}} + D_k B_{k-1}\|f\|_{W^{k-2,1}} \\
&\leq \alpha^{\frac{k-1}{2}2n}\left(C_{k+1} + D_k\right)\|f\|_{W^{k,1}} +
D_k D_{k-1}\|f\|_{W^{k-2,1}}.
\end{align*}
And in the case $p=n+1$ (an odd exponent), we have 
\begin{align*}
\|L^{2n+1}_\delta f\|_{W^{k,1}} &\leq \alpha^{\frac{k-1}{2}%
(2n+1)}\left(\alpha^{n-\frac{k-1}{2}}C_{k+1} + \alpha^{\frac{k-1}{2}%
}D_k\right)\|f\|_{W^{k,1}} + D_k D_{k-1}\|f\|_{W^{k-2,1}} \\
&\leq \alpha^{\frac{k-1}{2}(2n+1)}\left(\alpha^{-\frac{k-1}{2}}C_{k+1} +
D_k\right)\|f\|_{W^{k,1}} + D_k D_{k-1}\|f\|_{W^{k-2,1}}.
\end{align*}
Taking the maximum of the two constants that differ finishes the proof:
because $\alpha <1$, we can take $E_k = \left(\alpha^{-\frac{k-1}{2}%
}C_{k+1}+D_k\right)$ and $F_k = D_k D_{k-1}$.
\end{proof}

Using the compact embedding of $W^{7,1}(\mathbb{S}^1)$ into $W^{5,1}(\mathbb{%
S}^1)$ (by Rellich-Kondrachov embedding theorem, see \cite{ADA}) and the
Lasota-Yorke inequality we just proved, one can easily deduce assumption LR1
(an example of such reasoning can be found in \cite{notes}).

\subsubsection{First derivative operator (verifying LR4 and QR1)\label%
{LR4QR1}}

In this subsection we prove LR4 and QR1. These assumptions concern the first
derivative operator. We will first prove that this first derivative operator
is zero.

\begin{lemma}
Let $(\rho_{\delta })_{\delta }$ be the family of random kernels defined at $%
($\ref{eq:kerdef}$)$. There exists $C\geq 0$ such that for all $f\in W^{2,1}(%
\mathbb{S}^1)$, the following inequality holds 
\begin{equation*}
\left\|\frac{\rho_\delta - \delta_0}{\delta}*f \right\|_{L^1} \leq \delta
\|f\|_{W^{2,1}} C.
\end{equation*}
\end{lemma}

\begin{proof}
Let us use the following Taylor expansion for $f\in W^{2,1}(\mathbb{S}^1)$: 
\begin{equation}  \label{eq:taylor1}
f(x-\delta z) = f(x) - \delta z f^{\prime }(x) - \int_{x-\delta z}^x
(x-\delta z - t)f^{\prime \prime }(t)dt.
\end{equation}
Using (\ref{eq:kernel}), we have 
\begin{equation*}
((\rho_\delta - \delta_0)*f)(x) = -\int_{-1}^1\int_{x-\delta z}^x
\rho(z)(x-\delta z-t)f^{\prime \prime }(t)dtdz
\end{equation*}
By using the substitution $y = \frac{x-t}{\delta}$ in the last integral, we
can re-write the result as 
\begin{equation*}
((\rho_\delta - \delta_0)*f)(x) = \delta^2 \int_{-1}^1\int_z^0
\rho(z)(y-z)f^{\prime \prime }(x-\delta y)dydz
\end{equation*}
In particular, 
\begin{align*}
\left\|\frac{\rho_\delta - \delta_0}{\delta}*f \right\|_{L^1} &\leq \delta
\int_{\mathbb{S}^1} \int_{-1}^1\int_{[0,z]} \rho(z)|z-y||f^{\prime
\prime}(x-\delta y)| dydzdx \\
&= \delta \int_{-1}^1\int_{[0,z]} \rho(z)|z-y| \left( \int_{\mathbb{S}^1}
|f^{\prime \prime }(x-\delta y)|dx \right) dydz \\
\left\|\frac{\rho_\delta - \delta_0}{\delta}*f \right\|_{L^1} &\leq \delta
\|f\|_{W^{2,1}} \underbrace{\int_{-1}^1\int_{[0,z]} \rho(z)|z-y| dydz}_{= C}
\end{align*}
\end{proof}

We use this lemma to prove QR1 with a zero first derivative operator.

\begin{proposition}
\label{1der} Let $(L_{\delta })_{\delta }$ be the family of operators
defined at $($\ref{opdef}$)$. The following limit, defining the first
derivative operator holds 
\begin{equation*}
\lim_{\delta \rightarrow 0}\left\Vert \frac{L_{\delta }-L_{0}}{\delta }%
\right\Vert _{W^{3,1}\rightarrow W^{1,1}}=0.
\end{equation*}
\end{proposition}

\begin{proof}
Let us consider $f\in W^{3,1}(\mathbb{S}^{1}),$ we get 
\begin{align*}
\left\Vert \frac{L_{\delta }-L_{0}}{\delta }f\right\Vert _{W^{1,1}}&
=\left\Vert \frac{\rho _{\delta }-\delta _{0}}{\delta }\ast
(L_{T}f)\right\Vert _{W^{1,1}} \\
& =\left\Vert \frac{\rho _{\delta }-\delta _{0}}{\delta }\ast
(L_{T}f)\right\Vert _{L^{1}}+\left\Vert \frac{\rho _{\delta }-\delta _{0}}{%
\delta }\ast (L_{T}f)^{\prime }\right\Vert _{L^{1}} \\
& \leq \delta C \left( \Vert L_{T}f \Vert_{W^{2,1}} + \Vert (L_{T}f)^{\prime
} \Vert_{W^{2,1}} \right) \\
& \leq 2\delta C\Vert L_{T}f\Vert _{W^{3,1}} \\
& \leq 2\delta C\Vert L_{T}\Vert _{W^{3,1}\rightarrow W^{3,1}}\Vert
f\Vert_{W^{3,1}}.
\end{align*}%
The operator norm is then bounded by $2\delta C \Vert L_T\Vert$, which tends
to 0 when $\delta$ does.
\end{proof}

To finish verifying assumption LR4: we can remark that $\forall k\geq 0$, 
\begin{align*}
\Vert (\rho _{\delta }-\delta _{0})\ast f\Vert _{W^{k,1}}
&=\sum_{i=0}^{k}\Vert \left( (\rho _{\delta }-\delta _{0})\ast
f\right)^{(i)}\Vert _{L^{1}} \\
& \leq \delta ^{2}C\sum_{i=0}^{k}\Vert f^{(i)}\Vert _{W^{2,1}} \\
& \leq \delta ^{2}C(k+1)\Vert f\Vert _{W^{k+2,1}} \\
\Vert (\rho _{\delta }-\delta _{0})\ast f\Vert _{W^{k,1}} & \leq \delta C
(k+1) \Vert f\Vert _{W^{k+2,1}}.
\end{align*}%
So $\forall k\geq 0$, 
\begin{equation*}
\Vert (L_{\delta }-L_{0})f\Vert _{W^{k,1}}\leq \delta C (k+1) A_{k+2}\Vert
f\Vert _{W^{k+2,1}}
\end{equation*}%
with $A_{k+2} $ the constant from Lemma \ref{Lemsu}. We then have LR4, with
the result for $k\in \{3,5\}$: 
\begin{equation*}
\Vert L_{\delta }-L_{0}\Vert _{W^{k+2,1}\rightarrow W^{k,1}}\leq \delta C
(k+1) A_{k+2}.
\end{equation*}

\subsubsection{Second derivative operator (verifying QR2)\label{QR2}}

In this subsection we prove assumption QR2, computing the second derivative
operator and showing its relation with the variance of $\rho $.

\begin{lemma}
Let $(\rho_\delta)_\delta$ be the family of random kernels described in \ref%
{eq:kerdef}. Then there exists $C >0$ such that for all $f \in W^{3,1}(%
\mathbb{S}^1)$, the following inequality holds 
\begin{equation*}
\left\|\frac{\rho_\delta - \delta_0}{\delta^2}*f - \frac{\sigma^2(\rho)}{2}
f^{\prime \prime }\right\|_{L^1} \leq \delta C \|f\|_{W^{3,1}}.
\end{equation*}
\end{lemma}

\begin{proof}
We can extend for $f\in W^{3,1}(\mathbb{S}^1)$ the Taylor expansion (\ref%
{eq:taylor1}): 
\begin{equation}  \label{eq:taylor2}
f(x-\delta z)=f(x)-\delta zf^{\prime }(x)+\frac{\delta ^{2}z^{2}}{2}%
f^{\prime \prime }(x)-\int_{x-\delta z}^{x}\frac{(x-\delta z-t)^{2}}{2}%
f^{(3)}(t)dt.
\end{equation}%
We then have 
\begin{equation*}
((\rho _{\delta }-\delta _{0})\ast f)(x)=\delta ^{2}\frac{f^{\prime \prime
}(x)}{2}\sigma^2(\rho )-\int_{-1}^{1}\int_{x-\delta z}^{x}\rho (z)\frac{%
(x-\delta z-t)^{2}}{2}f^{(3)}(t)dtdz.
\end{equation*}

By using the substitution $y = \frac{x-t}{\delta}$ in the last integral, we
can re-write the result as 
\begin{equation*}
((\rho_\delta - \delta_0)*f)(x) = \delta^2 \frac{f^{\prime \prime }(x)}{2}%
\sigma^2(\rho) + \frac{\delta^3}{2}\int_{-1}^1\int_z^0
\rho(z)(z-y)^2f^{(3)}(x-\delta y)dydz
\end{equation*}
i.e. 
\begin{equation*}
\left(\frac{\rho_\delta - \delta_0}{\delta^2}*f\right)(x) = \frac{f^{\prime
\prime }(x)}{2}\sigma^2(\rho) + \frac{\delta}{2}\int_{-1}^1\int_z^0
\rho(z)(z-y)^2f^{(3)}(x-\delta y)dydz
\end{equation*}

Once again we can use Fubini theorem to exchange the last integrals: 
\begin{equation*}
\int_{-1}^1\int_z^0\rho(z)(z-y)^2f^{(3)}(x-\delta y)dydz = R_1(x,\delta) +
R_2(x,\delta)
\end{equation*}
with 
\begin{align*}
R_1(x,\delta) &= \int_{-1}^0\int_z^0\rho(z)(z-y)^2f^{(3)}(x-\delta y)dydz \\
&=\int_{-1}^0 f^{(3)}(x-\delta y) \left( \int_{-1}^y\rho(z)(z-y)^2 dz
\right)dy
\end{align*}
and 
\begin{align*}
R_2(x,\delta) &= \int_0^1\int_z^0\rho(z)(z-y)^2f^{(3)}(x-\delta y)dydz \\
&=\int_0^1 f^{(3)}(x-\delta y) \left(\int_y^1\rho(z)(z-y)^2 dz \right)dy.
\end{align*}
So 
\begin{equation*}
\left(\frac{\rho_\delta - \delta_0}{\delta^2}*f\right)(x) = \frac{f^{\prime
\prime }(x)}{2}\sigma^2(\rho) + \frac{\delta}{2}\int_{-1}^1 f^{(3)}(x-\delta
y) \Omega(y)dy
\end{equation*}
with 
\begin{equation*}
\Omega(y) = \left\{ 
\begin{array}{ll}
\int_y^1 -\rho(z)(z-y)^2dz & \text{if } y\geq 0 \\ 
\int_{-1}^y \rho(z)(z-y)^2dz & \text{if } y < 0%
\end{array}
\right. .
\end{equation*}

We can then conclude that 
\begin{align*}
\left\|\frac{\rho_\delta - \delta_0}{\delta^2}*f - \frac{\sigma^2(\rho)}{2}
f^{\prime \prime }\right\|_{L^1} &\leq \frac{\delta}{2} \int_{\mathbb{S}%
^1}\int_{-1}^1 \left|f^{(3)}(x-\delta y) \Omega(y) \right|dy dx \\
&\leq \frac{\delta}{2} \int_{-1}^1 |\Omega(y)| dy \times \|f\|_{W^{3,1}}.
\end{align*}
\end{proof}

As in Subsection \ref{LR4QR1}, we can apply this lemma to our problem,
obtaining the following.

\begin{proposition}
Suppose $T$ is a $C^{5}$ expanding map on the circle $\mathbb{S}^{1}$. Let $%
h_{0}\in \mathbb{S}^{1}$, its invariant probability density and let $%
L_{\delta }$ be the family of operators defined in $($\ref*{opdef}$)$ then
the following holds 
\begin{equation*}
\left\Vert \frac{(L_{\delta }-L_{0})h_{0}}{\delta ^{2}}-\frac{\sigma^2(\rho )%
}{2}h_{0}^{\prime \prime }\right\Vert _{W^{1,1}}\underset{\delta \rightarrow
0}{\longrightarrow }0.
\end{equation*}
\end{proposition}

\begin{proof}
Remark that because $h_0$ is the invariant probability measure of $T$ and
the property of derivation of a convolution product, 
\begin{equation*}
\frac{(L_{\delta }-L_{0})h_{0}}{\delta ^{2}} = \frac{\rho_{\delta }-\delta
_{0}}{\delta ^{2}}\ast h_{0} \qquad \text{and} \qquad \left(\frac{(L_{\delta
}-L_{0})h_{0}}{\delta ^{2}}\right)^\prime = \frac{\rho_{\delta }-\delta _{0}%
}{\delta ^{2}}\ast h^\prime_{0}.
\end{equation*}

$T$ being a $C^{5}$ expanding map imply that $h_0$ is $C^4$ (see \cite{notes}%
): we can then apply our lemma to both $h_0$ and $h_0^\prime$, giving us 
\begin{equation*}
\left\Vert \frac{(L_{\delta }-L_{0})h_{0}}{\delta ^{2}}-\frac{\sigma^2(\rho )%
}{2}h_{0}^{\prime \prime }\right\Vert _{L^{1}} \leq \delta C
\|h_0\|_{W^{3,1}}
\end{equation*}
and 
\begin{equation*}
\left\Vert \left(\frac{(L_{\delta }-L_{0})h_{0}}{\delta ^{2}}-\frac{%
\sigma^2(\rho )}{2}h_{0}^{\prime \prime }\right)^\prime\right\Vert _{L^{1}}
\leq \delta C \|h^\prime_0\|_{W^{3,1}}.
\end{equation*}
The result then follows.
\end{proof}

We have then verified the assumption QR2. Since all the assumptions are
verified, we can hence apply Theorem \ref{thm:quadresp} to the family of
perturbed operators $L_{\delta }$, proving Theorem \ref{thm:quadrespsmooth}.

\begin{proof}[Proof of Theorem \protect\ref{thm:quadrespsmooth}]
We apply Theorem \ref{thm:quadresp}, with the spaces $B_i = W^{i,1}(\mathbb{S%
}^1)$, with $i \in \{ss,s,w,ww\} = \{7,5,3,1\}$. We showed that our family
of operator verifies the assumptions of both Theorem \ref{thm:linresp} and %
\ref{thm:quadresp} in the previous subsections: assumptions LR2, LR3 and QR3
in subsection \ref{LR2LR3QR3}, LR1 in subsection \ref{LR1}, LR4 and QR1 in
subsection \ref{LR4QR1}, and QR2 in subsection \ref{QR2}.
\end{proof}

\section{Quantitative zero-noise limit of piecewise expanding maps\label%
{sec:pwexp}}

In this section we prove that for a certain family of piecewise expanding
maps, the invariant densities in the zero-noise limit have a speed of
convergence \ "of order at least about 1", as stated more precisely in
Proposition \ref{prop:linresppiecewise}, confirming the numerical findings
of \cite{Lin}. In this paper, the author shows numerically one example of
piecewise expanding map having a discontinuous invariant density, where the
speed of convergence is of order 1. This is due to the presence of
discontinuities in the map and in the corresponding invariant densities. We
remark that, as shown in the previous section, the exponent can be larger
than 1 for smoother maps. The proof of Proposition \ref%
{prop:linresppiecewise} is composed of three parts: in section \ref{UFO} we
introduce the concept of Uniform Family of Operators and state their link
with the speed of convergence to equilibrium . We then show that the family
of perturbations we consider in the small noise limit is uniform in this
sense. Finally, in section \ref{LB} we show a lower bound on the speed of
convergence, based on the the approximation of a discontinuity by Lipschitz
functions.

\subsection{Upper bounds: Convergence to equilibrium and stability\label{UB}}

In this section we provide the upper bounds sufficient to prove Proposition %
\ref{prop:linresppiecewise}. We start by defining the class of maps we mean to
consider.

\begin{definition}
A map $T:\mathbb{S}^{1}\rightarrow \mathbb{S}^{1}$ is said to be piecewise $%
C^{2}$ if there exists a finite set of points $d_{1}=0<d_{2}<...<d_{n}=1$
such that for each $0\leq i<n$, \ $T_{i}:=T_{(d_{i},d_{i+1})}$ extends to a $%
C^{2}$ function on the closure. Its expanding constant is defined as $%
\lambda _{T}=\inf_{i,x\in \lbrack d_{i},d_{i+1}]}\left\vert T^{\prime
}(x)\right\vert $.

A piecewise $C^2$ map is called piecewise expanding if there is a integer $k
>0$ such that $\lambda_{T^k} >1$, where $T^k$ is the $k^{th}$ iterate of $T$.
\end{definition}

\begin{definition}
A turning point of a map $T$ is a point where the derivative of the map is
not well defined.
\end{definition}

\subsubsection{Uniform Family of Operators, exponential convergence to
equilibrium and quantitative statistical stability\label{UFO}}

In this subsection we present a general quantitative result relating the
stability of the invariant measure of an uniform family of operator and the
speed of convergence to equilibrium.

Let $L$ be a Markov operator acting on two vector subspaces of signed
measures on $\mathbb{S}^{1}$, $L:(B_{s},||~||_{s})\longrightarrow
(B_{s},||~||_{s})$ and $L:(B_{w},||~||_{w})\longrightarrow (B_{w},||~||_{w})$%
, endowed with two norms, $||~||_{s}$ on $B_{s}, $ and $||~||_{w}$ on $B_{w}$%
, such that $||~||_{s}\geq ||~||_{w}$. Suppose that, 
\begin{equation*}
B_{s}\mathcal{\subseteq }B_{w}\mathcal{\subseteq }BS(\mathbb{S}^{1}),
\end{equation*}%
where again $BS(\mathbb{S}^{1})$ denotes the space of Borel finite signed
measures on $\mathbb{S}^{1}.$ Let us consider again the space of zero
average measures%
\begin{equation}
V_{s}=\{f\in B_{s},f(\mathbb{S}^{1})=0\}.  \label{vs}
\end{equation}%
This space is preserved by any Markov opertator.

We say that $L$ has convergence to equilibrium with at least speed $\Phi $
and with respect to the norms $||~||_{s}$ and $||~||_{w}$, if for each $f\in
V_{s}$ it holds 
\begin{equation}
||Lf||_{w}\leq \Phi (n)||f||_{s},  \label{wwe}
\end{equation}%
where $\Phi (n)\longrightarrow 0$ as $n\longrightarrow \infty $.

\begin{definition}
\label{UF} A one parameter family of transfer operators $\{L_{\delta
}\}_{\delta \in \left[ 0,1\right) }$ is said to be an \textbf{uniform family
of operators} with respect to the weak space $(B_{w},||~||_{w})$ and the
strong space $(B_{s},||~||_{s})$ if $||~||_{s}\geq ||~||_{w}$ and it
satisfies

\begin{enumerate}
\item[\textbf{UF1}] Let $h_{\delta }\in B_{s}$ be a probability measure
fixed under the operator $L_{\delta }$. Suppose there is $M>0$ such that for
all $\delta \in \lbrack 0,1)$, it holds 
\begin{equation*}
||h_{\delta }||_{s}\leq M;
\end{equation*}

\item[\textbf{UF2}] $L_{\delta }$ approximates $L_{0}$ when $\delta $ is
small in the following sense: there is $C\in \mathbb{R}^{+}$ such that: 
\begin{equation*}
||(L_{0}-L_{\delta })h_{\delta }||_{w}\leq \delta C;
\end{equation*}

\item[\textbf{UF3}] $L_{0}$ has exponential convergence to equilibrium with
respect to the norms $||~||_{s}$ and $||~||_{w}$: there exists $0<\rho
_{2}<1 $ and $C_{2}>0$ such that 
\begin{equation*}
\forall \ f\in V_{s}:=\{f\in B_{s}:f(X)=0\}
\end{equation*}%
it holds 
\begin{equation*}
||L_{0}^{n}f||_{w}\leq \rho _{2}^{n}C_{2}||f||_{s};
\end{equation*}

\item[\textbf{UF4}] The iterates of the operators are uniformly bounded for
the weak norm: there exists $M_{2}>0$ such that 
\begin{equation*}
\forall \delta ,n,g\in B_{s}\ \mathnormal{it~holds}\ ||L_{\delta
}^{n}g||_{w}\leq M_{2}||g||_{w}.
\end{equation*}
\end{enumerate}
\end{definition}

Under these assumptions we can ensure that the invariant measure of the
system varies continuously (in the weak norm) when $L_{0}$ is perturbed to $%
L_{\delta }$, for small values of $\delta $. Moreover, the modulus of
continuity can be estimated. The following result was indeed proved in \cite%
{GLu}.

\begin{proposition}
\label{dlogd}Suppose $\{\func{L}_{\delta }\}_{\delta \in \left[0, 1 \right)}$
is an uniform family of operators as in Definition \ref{UF}, where $h_{0}$
is the unique fixed point of $\func{L}_{0}$ in $B_{w}$ and $h_{\delta }$ is
a fixed point of $\func{L} _{\delta }$. Then, there exists $\delta _0 \in
(0,1)$ such that for all $\delta \in [0,\delta _0)$, it holds

\begin{equation*}
||h_{\delta }-h_{0}||_{w}=O(\delta \log \delta ).
\end{equation*}
\end{proposition}

It is worth to remark that such a statement can be generalized to other
speed of convergence to equilibrium, obtaining for example Holder bounds to
the statistical stability of systems having a power law speed of convergence
to equilibrium (see \cite{Gcsf},\cite{Gpre}).

In the next section, we will prove that our small noise perturbation gives
us a uniform family of operator. We can then apply Proposition \ref{dlogd}
to our family to prove an upper bound on the speed of convergence of
invariant densities. Note that it does not give us a purely linear upper
bound $O(\delta)$; however a convergence in $\delta \log \delta$ (up to a
multiplicative constant) would still give an exponent 1 if extracted as a
power law behavior: 
\begin{equation*}
\lim_{\delta \to 0} \frac{\log \|h_\delta - h_0\|_1}{-\log(\delta)} = 1.
\end{equation*}

\subsection{Proof that the small noise perturbation gives a uniform family
of operators}

\subsubsection{UF3 and UF4}

Assumption \textbf{UF4} is immediate, as transfer operators are contractions
on $L^1$. As showed earlier, we have that for all $f \in L^1$, \ $%
\|\rho_\delta * f\|_{L^1} \leq \|f\|_{L^1}$. $L_T$ being a contraction on $%
L^1$, we then have 
\begin{equation*}
\|L_\delta f\|_{L^1} \leq \|L_T f\|_{L^1} \leq \|f\|_{L^1}.
\end{equation*}
$L_\delta$ is then also a contraction on $L^1$, hence the result for all $n$%
: $\|L^n_\delta f\|_{L^1} \leq \|f\|_{L^1}$.

Assumption \textbf{UF3} is verified for our spaces $BV$ and $L^1$ for
piecewise expanding maps of the circle (see \cite{notes}).

\subsubsection{UF2}

We first prove a similar result, but only for smooth functions. The
calculations are basically the same as the ones we had for the derivative
operator in the smooth expanding maps case.

\begin{lemma}
\label{lem:Cinftyresult} There exists a $C >0$ such that for all $f \in
C^\infty$, 
\begin{equation*}
\|\rho_\delta*f - f\|_{L^1} \leq C\|f^{\prime }\|_{L^1} \delta .
\end{equation*}
\end{lemma}

\begin{proof}
\begin{align*}
\|\rho_\delta \ast f - f\|_{L^1} &= \int_{\mathbb{S}^1} \frac{1}{\delta}%
\left| \int_{-\delta}^\delta \rho\left(\frac{y}{\delta}\right)[f(x-y) -
f(x)]dy \right| dx \\
&\leq \int_{-\delta}^\delta \frac{1}{\delta}\rho\left(\frac{y}{\delta}%
\right) \left( \int_{\mathbb{S}^1} \left| \int_y^0 f^{\prime }(x-t)dt
\right| dx \right) dy \\
&\leq \|f^{\prime }\|_1 \times \int_{-\delta}^\delta \frac{1}{\delta}%
\rho\left(\frac{y}{\delta}\right)|y| dy \\
\|\rho_\delta \ast f - f\|_{L^1} &\leq \|f^{\prime }\|_{L^1} \times \delta
\times \underbrace{\int_{-1}^1\rho(z)|z| dz}_{=C}
\end{align*}
\end{proof}

To extend the result for all $BV$ functions, we will use the following
lemma, which proof can be found in \cite{EG}. The set of smooth functions
are not dense in $BV$ for their norm; however we can still approximate $BV$
functions by smooth ones in a weaker sense.

\begin{lemma}
\label{lem:smoothApprox} For all $f \in BV$, there exists a sequence $%
(f_n)\in (C^\infty \cap BV)^{\mathbb{N}}$ such that 
\begin{equation*}
\left\{%
\begin{array}{ll}
\|f-f_n\|_{L^1} & \underset{n\rightarrow \infty}{\longrightarrow} 0 \\ 
\Var(f_n) & \underset{n\rightarrow \infty}{\longrightarrow} \Var(f)%
\end{array}
\right.
\end{equation*}
\end{lemma}

We can then extend our result from Lemma \ref{lem:Cinftyresult}.

\begin{proposition}
\label{prop:BVresult} There exists a $C >0$ such that for all $f \in BV$, 
\begin{equation*}
\|\rho_\delta*f - f\|_{L^1} \leq C\Var(f) \delta .
\end{equation*}
\end{proposition}

\begin{proof}
Lemma \ref{lem:Cinftyresult} gives us the result for all $f \in C^\infty$.
Indeed, for them $\Var(f) = \|f^{\prime }\|_{L^1}$. Now let $g$ be a $BV$
function, and $\epsilon >0$ be arbitrarily small. Let us prove 
\begin{equation*}
\|\rho_\delta*g - g\|_{L^1} \leq C \Var(g)\delta +\epsilon.
\end{equation*}
with $C$ the same constant from Lemma \ref{lem:Cinftyresult}.

Using Lemma \ref{lem:smoothApprox}, there is a $f\in C^\infty$ such that 
\begin{equation*}
\left\{ 
\begin{array}{l}
\|f-g \|_{L^1} \leq \frac{\epsilon}{3} \\ 
\Var(f) \leq \Var(g) + \frac{\epsilon}{3C\delta}%
\end{array}
\right.
\end{equation*}
We then have 
\begin{align*}
\|\rho_\delta*g - g\|_{L^1} &\leq \|\rho_\delta*(g - f)\|_{L^1} +
\|g-f\|_{L^1} + \|\rho_\delta*f -f\|_{L^1} \\
&\leq \frac{\epsilon}{3} + \frac{\epsilon}{3} + C\|f^{\prime }\|_{L^1}\delta
\\
&= \frac{2\epsilon}{3} + C\Var(f)\delta \\
\|\rho_\delta*g - g\|_{L^1} &\leq \epsilon + C\Var(g)\delta
\end{align*}
\end{proof}

We then have \textbf{UF2} assuming \textbf{UF1}: indeed, because $L_T
h_\delta \in BV$, we can use the result from Proposition \ref{prop:BVresult}
as: 
\begin{align*}
\|(L_\delta - L_0)h_\delta\|_{L^1} &= \|\rho_\delta* L_T h_\delta - L_T
h_\delta\|_{L^1} \\
&\leq C\Var(L_T h_\delta) \delta \\
&\leq C |||L_T|||_{BV \to BV} \|h_\delta\|_{BV} \delta \\
\|(L_\delta - L_0)h_\delta\|_1 &\leq C |||L_T|||_{BV \to BV} M \delta
\end{align*}
with $M$ the constant of \textbf{UF1}.

\subsubsection{UF1}

To prove the strong boundedness of the family of BV functions $\{h_\delta\}$%
, we can use a uniform L-Y inequality on $L_\delta$.

\begin{remark}
In the first section, the unperturbed operator $L_0$ verified a Lasota-Yorke
inequality of type 
\begin{equation*}
\|L_0^n f\|_s \leq \alpha^n \|f\|_s + C \|f\|_w,
\end{equation*}
where the proof is based on iterating the case $n=1$: we proved the uniform
Lasota-Yorke inequality using $\|L_\delta f\| \leq \|L_T f\|$. In the
general case of a piecewise expanding map, we have (see \cite{notes}) 
\begin{equation*}
\|L_0^n f\|_s \leq \alpha^n A \|f\|_s + B \|f\|_w.
\end{equation*}
Proving a uniform Lasota-Yorke inequality is then more complex.
\end{remark}

We use the result stated in \cite{BK2}, which was proved in \cite{BK1}.

\begin{definition}
A transition probability is a linear positive (sub)-Markovian operator $Q:
L^1 \to L^1$ such that $\|Q\|_1 \leq 1$. $Q^*$ denote its dual operator on $%
L^\infty$. A transition probability can be represented via a (sub)-Markov
transition kernel on $[0,1]$ into itself: 
\begin{equation*}
Q^*h(x) = \int h(y)Q(x,dy) \qquad \text{and} \qquad Qh(y) = \left( \frac{d}{%
dm}\int h(x) Q(x,.)m(dx)\right)(y).
\end{equation*}

If $Q(x,.) \ll m$ for each $x$, we note $q(x, y) = \frac{d}{dm}Q(x,.)(y)$.
\end{definition}

\begin{proposition}
\label{prop:BlankKeller} Let $T$ be a piecewise expanding map with no
periodic turning point. Suppose it is perturbed by a family of transition
probabilities $\{Q_\delta\}_\delta$ (i.e. $L_\delta = Q L_T$) verifying the
following assumptions: 
\begin{equation}
\text{(Small perturbation) } \mathbf{d}(Q_\delta) := \sup\{\|Q_\delta f -
f\|_{L^1} \mid \|f\|_{BV}\leq 1\} \underset{\delta\rightarrow 0}{%
\longrightarrow} 0.
\end{equation}
\begin{equation}
\text{(Locality) }\forall x,A \text{ s.t. dist}(x,A) > \delta, \quad
Q_\delta(x,A) = 0
\end{equation}
\begin{equation}
\text{(Regularity) } \forall f \in BV, \quad \Var(Q_\delta f) \leq \Var(f) +
C\|f\|_{L^1}
\end{equation}
where $Q_\delta f$ represent the density of $A \mapsto \int
Q_\delta(x,A)f(x)dx$ with respect to the Lebesgue measure.

Then there exists constants $C, \delta_0, \alpha <1$ and $N\in {\mathbb{N}}$
such that 
\begin{equation*}
\Var(L_\delta^N f) \leq \alpha \Var(f) + C \|f\|_1
\end{equation*}
$\forall \delta \leq \delta_0$ and $f\in BV$.
\end{proposition}

In the case of an additive noise $\rho _{\delta }$, the Markov kernels have
densities $q_{\delta }$ with respect to the Lebesgue measure, defined as
(with the subtraction on $\mathbf{S}^1$) 
\begin{equation*}
Q_{\delta }(x,A)=\int_{A}q_{\delta }(x,y)dy\qquad \text{with}\qquad
q_{\delta }(x,y):=\rho _{\delta }(y-x).
\end{equation*}%
Then $Q_{\delta }f=\rho _{\delta }\ast f$. The \textit{small perturbation}
assumption is a simple application of Proposition \ref{prop:BVresult}, as it
gives us that $\mathbf{d}(Q_{\delta })\leq C\delta \underset{\delta
\rightarrow 0}{\longrightarrow }0$. The \textit{locality} assumption is
verified as the support of $\rho _{\delta }$ is included in the interval $%
[-\delta ,+\delta ]$. The \textit{regularity} assumption is easily verified
by our noise kernel via the following lemma.

\begin{lemma}
$\forall \delta$ and $f\in BV$, 
\begin{equation*}
\Var(\rho_\delta * f) \leq \Var(f)
\end{equation*}
\end{lemma}

\begin{proof}
One equivalent definition of Var is the following: 
\begin{equation*}
\Var(f) = \sup \left\{ \int_{\mathbb{S}^1} \phi^{\prime}(x) f(x) dx ~\mid
\phi \in C^1_c \text{ s.t. } \|\phi\|_\infty \leq 1\right\}
\end{equation*}
Let $\phi \in C^1_c$. We then have 
\begin{align*}
\int \phi^{\prime }\times (\rho_\delta*f) dx &= \int \int_{-\delta}^\delta 
\frac{1}{\delta}\rho\left(\frac{y}{\delta}\right)\phi^{\prime }(x)f(x-y) dydx
\\
&= \int_{-\delta}^\delta \frac{1}{\delta}\rho\left(\frac{y}{\delta}%
\right)\left( \int\phi^{\prime }(x)f(x-y) dx\right) dy \\
&= \int_{-\delta}^\delta \frac{1}{\delta}\rho\left(\frac{y}{\delta}\right)
\left( \int\phi^{\prime }(\tilde{x}+y)f(\tilde{x}) d\tilde{x}\right) dy \\
&\leq \int_{-\delta}^\delta \frac{1}{\delta}\rho\left(\frac{y}{\delta}%
\right) \Var(f) dy \\
\int \phi^{\prime }\times (\rho_\delta*f) dx &\leq \Var(f).
\end{align*}
Hence $\Var(\rho_\delta*f) \leq \Var(f)$.
\end{proof}

Because our noise verifies all the assumptions, we can apply Proposition \ref%
{prop:BlankKeller}. Using the contracting property of $L_\delta$ on $L^1$,
we easily deduce the following.

\begin{proposition}
\label{prop:LYInequalityTurning} Let $T$ be a piecewise expanding map with
no periodic turning point. Then there exists constants $C, \overline{\delta}%
, \alpha <1$ and $N\in {\mathbb{N}}$ such that 
\begin{equation*}
\|L_\delta^N f\|_{BV} \leq \alpha \|f\|_{BV} + C^{\prime }\|f\|_{L^1}
\end{equation*}
$\forall \delta \leq \overline{\delta}$ and $f\in BV$.
\end{proposition}

This can give us an uniform L-Y inequality.

\begin{proposition}
Under the same assumptions as before, we have that $\forall p\in {\mathbb{N}}%
, \newline
\delta \leq \overline{\delta}, \ 0 \leq k < N, \ f\in BV$, 
\begin{equation*}
\|L_\delta^{pN+k}f\|_{BV} \leq \alpha^p |\|L_T|\|^k_{BV \to BV} \|f\|_{BV} + 
\frac{C}{1-\alpha}\|f\|_{L^1}
\end{equation*}
which then leads to 
\begin{equation*}
\|L_\delta^n f\|_{BV} \leq \alpha^n A \|f\|_{BV} + B\|f\|_{L^1} \qquad
\forall n\in {\mathbb{N}}.
\end{equation*}
\end{proposition}

\begin{proof}
The previous proposition gives us 
\begin{equation*}
\|L_\delta^N f\|_{BV} \leq \alpha \|f\|_{BV} + C^{\prime }\|f\|_{L^1}.
\end{equation*}
Using the same type of induction as in the proof of Lemma \ref{lem:LYNoise},
we have the following result $\forall p\in {\mathbb{N}}$: 
\begin{equation*}
\|L_\delta^{pN}f\|_{BV} \leq \alpha^p \|f\|_{BV} + C\sum_{i=0}^{p-1}
\alpha^i\|L_\delta^{p-1-i}f\|_{L^1} \leq \alpha^p \|f\|_{BV} + \frac{C}{%
1-\alpha}\|f\|_{L^1}.
\end{equation*}

Note that using the regularity assumption on our noise, we have that for all 
$f\in BV$, $\|L_\delta f \|_{BV} \leq \|L_T f\|_{BV}$. Then, $\forall k < N$%
, $\|L_\delta^k f \|_{BV} \leq |\|L_T|\|_{BV\to BV}^k \|f\|_{BV}$. We can
then conclude that 
\begin{align*}
\|L_\delta^{pN+k}f\|_{BV} &\leq \alpha^p \|L_\delta^k f\|_{BV} + \frac{C}{%
1-\alpha}\|L_\delta^k f\|_{L^1} \\
&\leq \alpha^p |\|L_\delta|\|_{BV\to BV}^k \|f\|_{BV} + \frac{C}{1-\alpha}%
\|f\|_{L^1}.
\end{align*}
Note that the final inequality is assuming that $|\|L_T|\| \geq 1$. If it is 
$\leq 1$, the constant is just $1$, but the L-Y inequality is then trivial:
using the previous lemma, we would also have the norm of $L_\delta$ being 1,
and have the inequality as in the smooth expanding case.
\end{proof}

Having proven a uniform Lasota-Yorke inequality, we can conclude that our
family of operators verifies also assumption $\mathbf{UF1}$. We then have
proved that the dynamics resulting from a piecewise expanding maps of the
circle with no periodic turning point perturbed by an additive noise have an
upper bound on their modulus of continuity. More explicitly, 
\begin{equation*}
\|h_\delta - h_0\|_{L^1} = O(\delta \log\delta).
\end{equation*}

\begin{remark}
The Lasota-Yorke inequality used in the verification of $\mathbf{UF1}$ might
be extended to piecewise expanding maps having periodic turning points by
the results in \cite{BK1}, therefore extending our conclusion to all
piecewise expanding maps of the circle.
\end{remark}

\subsection{Lower bounds: approximation of a discontinuity by Lipschitz
functions\label{LB}}

Until now, we only proved an upper bound on the modulus of continuity. Here,
we show examples of piecewise expanding map of the circle for which the
speed of approximation in the zero-noise limit is in fact of order $1$,
providing the lower bound sufficient to prove Proposition \ref%
{prop:linresppiecewise}. Let us consider the following map%
\begin{equation}
T:x\mapsto \left\{ 
\begin{array}{ll}
x+\frac{1}{2} & 0\leq x\leq \frac{1}{2} \\ 
2(1-x) & \frac{1}{2}\leq x\leq 1%
\end{array}%
\right. .  \label{mapT}
\end{equation}%
One gets easily that $T$ has the following invariant density, which is
discontinuous: 
\begin{equation}
h_{0}:x\mapsto \left\{ 
\begin{array}{ll}
\frac{2}{3} & ~0\leq x\leq \frac{1}{2} \\ 
\frac{4}{3} & ~\frac{1}{2}< x\leq 1%
\end{array}%
\right. .  \label{densH}
\end{equation}%
This example has already been studied in \cite{Lin}, where the author
numerically found linear speed of convergence in the zero-noise limit. Note
that $T$ admits $\{0, \frac{1}{2}, 1\}$ as periodic turning points, we
cannot apply the upper bound result proven in the previous section.

In this section we prove the following proposition

\begin{proposition}
\label{propapprox}Let $T$ be the map defined in $($\ref{mapT}$)$ and $h_{0}$
be its invariant density, as in $(\ref{densH}).$ Let $L_{\delta }$ be the
annealed transfer operator of the system with noise as defined at $(\ref%
{opdef})$ with $\rho $ $\in BV[-1,1]$ and let $h_{\delta }\in L^{1}$ be an
invariant density for $L_{\delta }$. Then there exists a constant $C\in 
\mathbb{R}$ such that 
\begin{equation*}
\Vert h_{\delta }-h_{0}\Vert _{L^{1}}\geq C\delta .
\end{equation*}
\end{proposition}

We prove the proposition by showing in section \ref{lipcon} that $h_{\delta
} $ is Lipschitz and providing an estimate for its Lipschitz constant,
showing that $h_{\delta }$ is $\frac{C^{\prime }}{\delta }$-Lipschitz for
some constant $C^{\prime }$. Then in section \ref{approx} we prove that
there is a $C^{\prime \prime }$ such that $\Vert f-h_{0}\Vert _{L^{1}}\geq 
\frac{C^{\prime \prime }}{a}$ for any function $f$ which is $a$-Lipschitz,
completing the proof.

\subsubsection{Estimating the Lipschitz constant of $h_{\protect\delta }$%
\label{lipcon}}

In this section we prove that under the assumptions proposition \ref%
{propapprox} for any $\delta >0$, the invariant density $h_{\delta }$ of the
perturbed system is $\frac{C^{\prime }}{\delta }$-Lipschitz. This will be
proved in Proposition \ref{proplip2}.

Before the main proposition we need two technical lemmas.

\begin{lemma}
For $f \in BV$, $h\geq 0$, we have 
\begin{equation*}
\int |f(x+h)-f(x)|dx \leq \Var(f)|h|.
\end{equation*}
\end{lemma}

\begin{proof}
We first prove it for $f \in C^\infty \cap BV$: 
\begin{align*}
\int |f(x+h)-f(x)|dx &= \int_\mathbb{R} \left| \int f^{\prime }(y) \chi_{x
\leq y \leq x+h} dy \right| dx \\
&\leq \int_\mathbb{R} |f^{\prime }(y)| dy \times h \\
&= \Var(f) \times h
\end{align*}

Then for all $f \in BV$: let us set $\epsilon >0$. Using Lemma \ref%
{lem:smoothApprox}, we can have $g \in C^\infty \cap BV$ such that $%
\|g-f\|_{L^1} \leq \epsilon$ and $\Var(g) \leq \Var(f) + \epsilon$. We then
have 
\begin{align*}
\int |f(x+h) - f(x)|dx &\leq \int |f(x+h) - g(x+h)| + |g(x) - f(x)| +
|g(x+h) - g(x)| dx \\
&\leq 2\epsilon + \Var(g)h \\
&\leq (2+h)\epsilon + \Var(f).
\end{align*}
We have the inequality for all $\epsilon >0$, so we have our result.
\end{proof}

\begin{lemma}
\begin{equation*}
\Var(\rho_\delta) = \frac{\Var(\rho)}{\delta}.
\end{equation*}
\end{lemma}

These two lemmas easily give us a Lipschitz constant for the convolution
product.

\begin{proposition}
\label{sop1}For all $f\in L^{\infty }$, the function $\rho _{\delta }\ast f$
is $\frac{\Var(\rho )\Vert f\Vert _{\infty }}{\delta }$-Lipschitz.
\end{proposition}

\begin{proof}
We use the first lemma to write that, for all $x\in \mathbb{S}^{1},h\geq 0$,%
\begin{align*}
|\rho _{\delta }\ast f(x+h)-\rho _{\delta }\ast f(x)|& \leq \int
|f(y)|\times |\rho _{\delta }(x-y+h)-\rho _{\delta }(x-y)|dy \\
& \leq \Vert f\Vert _{\infty }\Var(\rho _{\delta })h.
\end{align*}%
The second lemma then allows us to conclude.
\end{proof}

We now want to use this result to bound the Lipschitz constant of $h_{\delta
}$, the invariant density of the perturbed system.

\begin{proposition}
\label{proplip2}There is a $C^{\prime }>0$ such that for all $\delta >0$,
the invariant density of the perturbed system $h_{\delta }$ is $\frac{%
C^{\prime }}{\delta }$-Lipschitz.
\end{proposition}

\begin{proof}
By definition, $h_{\delta }=L_{\delta }h_{\delta }=\rho _{\delta }\ast
L_{T}h_{\delta }$. Proposition \ref{sop1} gives us that $h_{\delta }$ is $%
\frac{\Var(\rho )}{\delta }\Vert L_{T}h_{\delta }\Vert _{\infty }$%
-Lipschitz. Another well known result is the existence of a constant $A>0$
such that for all $f\in BV(\mathbb{S}^1), \|f\|_\infty \leq A \|f\|_{BV}$.
Hence: 
\begin{align*}
\Vert L_{T}h_{\delta }\Vert _{\infty }& \leq A\Vert L_{T}h_{\delta
}\Vert_{BV} &  \\
& \leq AB\Vert h_{\delta }\Vert _{BV} & & \text{because }L_{T}\text{ is
bounded on }BV \\
& \leq ABM & & \text{by property UF1 proven earlier}.
\end{align*}%
We then have our result, as all the constants are independent from $\delta $.
\end{proof}

\subsubsection{Approximation of a discontinuity \label{approx}}

We prove here the lower bound on the approximation of $h_0$ by $a$-Lipschitz
functions, with $a>0$ fixed.

Recall that $h_0$ is defined as (Figure \ref{fig:zeroslope}) 
\begin{equation*}
h_{0}:x\mapsto \left\{ 
\begin{array}{ll}
\frac{2}{3} & ~0\leq x\leq 0.5 \\ 
\frac{4}{3} & ~0.5< x\leq 1%
\end{array}%
\right. .
\end{equation*}

The intuitive "best approximation" function that is $a$-Lip would then be
the linear path, 
\begin{equation*}
f_{a}:x\mapsto \left\{ 
\begin{array}{cl}
\frac{2}{3} & \text{if }~x\leq 0.5-\frac{1}{3a} \\ 
1 + ax - \frac{a}{2} & \text{if }~0.5-\frac{1}{3a}\leq x\leq 0.5+\frac{1}{3a}
\\ 
\frac{4}{3} & \text{if }~x\geq 0.5+\frac{1}{3a}%
\end{array}%
\right. .
\end{equation*}%
We now prove that this is the best approximation in $L^{1}$, in the sense of
the following proposition.

\begin{figure}[h]
\caption{Lipschitz approximation of a discontinuity, graphical
representation of $h_0$ and $f_a$ ($a=3$).}
\label{fig:zeroslope}%
\begin{tikzpicture}
	\begin{axis}[legend pos=north west]
	\addplot[domain = 0:0.5, samples = 100]{2/3};
	\addplot[domain = 0.5:1, samples = 100]{4/3};
	\addplot[domain = 0:1, samples = 200, color = red]{(3*x-7/6)*(0.5-1/9 < x)*(x < 0.5+1/9) + 2/3 + (2/3)*(x > 0.5+1/9)};
	\legend{$h_0$,,$f_a$};
	\end{axis}
\end{tikzpicture}
\end{figure}
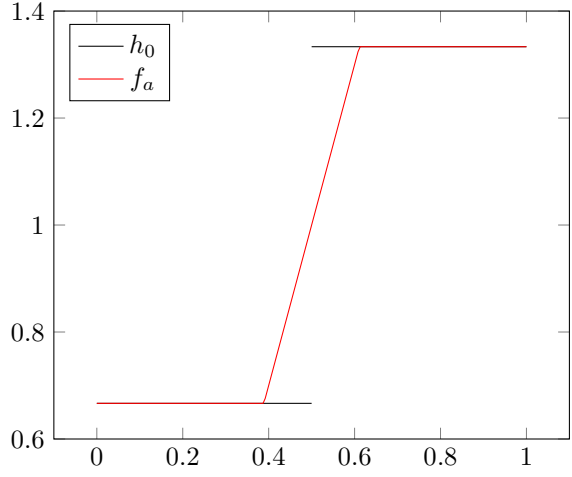

\begin{proposition}
Let $f$ be a real-valued $a$-Lipschitz function of $[0,1]$. The following
inequality holds: 
\begin{equation*}
\|f-h_0\|_{L^1} \geq \|f_a-h_0\|_{L^1} = \frac{1}{9a}.
\end{equation*}
\end{proposition}

\begin{proof}
The first step is to only consider the difference in the neighborhood of the
discontinuity where $f_a \neq h_0$: for all real-valued $f$, 
\begin{equation*}
\|f-h_0\|_{L^1[0,1]} \geq \|f-h_0\|_{L^1([0.5-\frac{1}{3a}, ~0.5+\frac{1}{3a}%
])}.
\end{equation*}

We can then simplify our problem by only considering functions with values
on the interval $[\frac{2}{3}, \frac{4}{3}]$. Indeed, for every real-valued $%
a$-Lip function $f$, if we denote by $\tilde{f}$ the function defined by $%
\tilde{f}:x \mapsto \min( \max(f(x),\frac{2}{3}), ~\frac{4}{3})$, the latter
is a better approximation of the discontinuity (in the sense $%
\|f-h_0\|_{L^1} \geq \|\tilde{f}-h_0\|_{L^1}$) while also being $a$%
-Lipschitz.

By a linear change of coordinates, one can see that proving the result on
the window $[0.5 - \frac{1}{3a},0.5+\frac{1}{3a}] \times [\frac{2}{3},\frac{4%
}{3}]$ for $a$-Lip functions is equivalent to proving it on $[0,1]\times[0,1]
$ for $1$-Lip functions, with $h_0$ and $f_a$ now being (Figure \ref%
{fig:zeroslope2}) 
\begin{equation*}
h_{0}:x\mapsto \left\{ 
\begin{array}{ll}
0 & ~0\leq x\leq 0.5 \\ 
1 & ~0.5< x\leq 1%
\end{array}%
\right. \qquad \text{and} \qquad f_a:x\mapsto x.
\end{equation*}

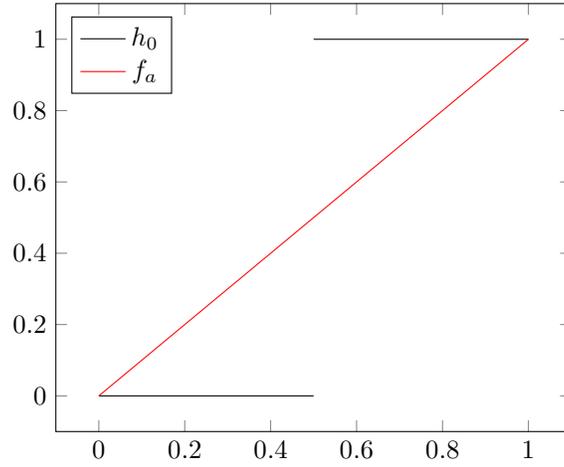
\begin{figure}[h]
\caption{Lipschitz approximation of a discontinuity, rescaling of the
problem.}
\label{fig:zeroslope2}%
\begin{tikzpicture}
	\begin{axis}[legend pos=north west]
	\addplot[domain = 0:0.5, samples = 100]{0};
	\addplot[domain = 0.5:1, samples = 100]{1};
	\addplot[domain = 0:1, samples = 100, color = red]{(1*x+0)*(0.5-1/2 < x)*(x < 0.5+1/2) + (x > 0.5+1/2)};
	\legend{$h_0$,,$f_a$};
	\end{axis}
\end{tikzpicture}
\end{figure}

Let $f$ be a $1$-Lip function of $[0,1]$, with values in $[0,1]$. 
\begin{equation*}
\|f-h_0\|_{L^1} = \int_{0}^{0.5} f(x) ~dx + \int_{0.5}^{1} 1-f(x) ~dx.
\end{equation*}
Using the $1$-Lip property, we have that for all $x > 0.5$, 
\begin{equation*}
f(x) - f(0.5) \leq \left|f(x) - f(0.5)\right| \leq x-0.5 \qquad \text{i.e.}
\qquad -f(x) \geq 0.5-x - f(0.5).
\end{equation*}
Hence 
\begin{align*}
\|f-h_0\|_{L^1} &\geq \int_{0}^{0.5} f(x) ~dx + \int_{0.5}^{1} 1-x ~dx + 
\frac{1}{2}(0.5 - f(0.5)) \\
&= \int_{0}^{0.5} f(x) + 0.5 - f(0.5) ~dx + \int_{0.5}^{1} 1-x ~dx.
\end{align*}
We can re-use the $1$-Lip property on $x < 0.5$ to obtain 
\begin{equation*}
f(x) - f(0.5) + 0.5 \geq x
\end{equation*}
and conclude 
\begin{equation*}
\|f-h_0\|_{L^1} \geq \int_{0}^{0.5} x ~dx + \int_{0.5}^{1} 1-x ~dx =
\|f_a-h_0\|_{L^1}.
\end{equation*}
\end{proof}

We are now ready to prove Proposition \ref{propapprox}.

\begin{proof}[Proof of Proposition \protect\ref{propapprox}]
We showed that there is a $C^{\prime }>0$ such that the invariant density of 
$L_{\delta }$ is $\frac{C^{\prime }}{\delta }$-Lipschitz. We can apply the
last proposition to state the following lower bound on the modulus of
continuity: 
\begin{equation*}
\Vert h_{\delta }-h_{0}\Vert _{L^{1}}\geq \frac{\delta }{9C^{\prime }}%
=C\delta .
\end{equation*}%
Note that this lower bound result could easily be applied to all piecewise
expanding maps with a discontinuity in their unperturbed invariant density,
with a different constant for each map.
\end{proof}

\newpage

\noindent \textbf{Acknowledgments.} S.G. is partially supported by the
research project PRIN 2017S35EHN\_004 "Regular and stochastic behavior in
dynamical systems" of the Italian Ministry of Education and Research. \ The
authors whish to thank Ecole Normale Paris Saclay and Universit\`{a} di Pisa
for the organization of the international master stage "Stage d'initiation 
\`{a} la recherche M1" in which framework the work has been done.
The authors also whish to thank W. Bahsoun and J. Sedro for useful discussions about zero noise limits and response.

\end{document}